\documentclass[11pt]{article}
\usepackage{amsfonts,amscd,amssymb, amsmath}
\newtheorem{definition}{Definition}[section]

\newtheorem{proposition}[definition]{Proposition}
\newtheorem{corollary}[definition]{Corollary}
\newtheorem{remark}[definition]{Remark}

\newtheorem{theorem}[definition]{Theorem}
\newtheorem{example}[definition]{Example}

\def\rawo\lonra{\longrightarrow}

\def\ot{\otimes}

\allowdisplaybreaks[4]

\newenvironment{proof}{{\it Proof.}}{\hfill $ \square $ \vskip 4mm}
\begin{document}
\title{Some (Hopf) algebraic properties of circulant matrices}
\author{Helena Albuquerque\thanks{The first author was partially supported
by the Centre for Mathematics of the University of Coimbra (CMUC)} \\
Departamento de Matem\'{a}tica, Universidade de Coimbra, \\
3001-454 Coimbra, Portugal\\
e-mail: lena@mat.uc.pt
\and Florin Panaite\thanks {The second author was partially supported by the
CNCSIS project ``Hopf algebras, cyclic homology and monoidal categories'',
contract nr. 560/2009, CNCSIS code $ID_{-}69$.}\\
Institute of Mathematics of the
Romanian Academy\\
PO-Box 1-764, RO-014700 Bucharest, Romania\\
e-mail: Florin.Panaite@imar.ro}
\date{}
\maketitle

\begin{abstract}
We study some (Hopf) algebraic properties of circulant matrices, inspired by
the fact that the algebra of circulant $n\times n$ matrices is isomorphic to the
group algebra of the cyclic group with $n$ elements. We introduce also a
class of matrices that generalize both circulant and skew circulant matrices, and
for which the eigenvalues and eigenvectors can be read directly from their entries.
\end{abstract}

\section*{Introduction}
${\;\;\;\;}$The starting point of this paper was a result in \cite{ak}, that
arose during the study of certain multiplicatively closed lattices and
so called Brandt algebras in (twisted) group rings of cyclic groups:
it asserts that for a twisted group ring ${\mathbb R}_F{\mathbb Z}_3$ (where $F$ is an
arbitrary map, not necessarily a two-cocycle) there exist three forms $q_1, q_2, q_3$,
concretely determined, such that any element $x\in {\mathbb R}_F{\mathbb Z}_3$
satisfies the polynomial equation $x^3-q_1(x)x^2+q_2(x)x-q_3(x)=0$. Moreover, $q_3(x)$ is
given by the determinant of a $3\times 3$ matrix (which, if $F$ is trivial, is a
circulant matrix) and, if $F$ is a two-cocycle, then $q_1$ and $q_2$ are related
in a certain (specific) way.

We wanted to see to what extent this kind of results may be generalized from $n=3$
to arbitrary $n$. It turns out that, even if analogous forms $q_1, ..., q_n$
may exist, not much could be said about them. Thus, we restricted our study
to the case when the map $F$ is trivial (that is, to ordinary group rings) and,
slightly more general, to the case when $F$ is a two-cocycle that is trivial in the
second cohomology group. We were thus led to consider circulant matrices as well as a
certain class of generalized circulants. By using circulants, we were able to prove that forms
$q_1, ..., q_n$ exist on the group ring ${\mathbb C}{\mathbb Z}_n$ and have some properties that
generalize the case $n=3$. This is done by using the well known algebra isomorphism
between the algebra of $n\times n$ circulants and the group algebra of the
cyclic group ${\mathbb Z}_n$. We found useful to give this result a Hopf
algebraic interpretation, obtaining along the way a result stating that the
algebra of $n\times n$ matrices "factorizes" (in a certain sense) as the
"product" between the algebra of $n\times n$ circulants and the algebra of $n\times n$ diagonal matrices.
Since the group ring is a Hopf algebra, the algebra of $n\times n$ circulants becomes also a Hopf algebra.
We have written down its Hopf structure and we found that the antipode looks particularly nice:
the antipode applied to a circulant matrix is simply the transpose of it. The
comultiplication $\Delta $ does not look too good, but we could
prove however that if $C$ is a circulant $n\times n$ matrix then $\Delta (C)$,
regarded as an $n^2\times n^2$ matrix, is block circulant with circulant blocks.
We present also a certain result (and a concrete example) concerning some lattices
in the algebra of circulant matrices.

In the last section of the paper we introduce a class of generalized circulants,
as follows. Denote the elements of ${\mathbb Z}_n$ by $e_1=1, e_2, ..., e_n$; for a
given map $\mu : {\mathbb Z}_n\rightarrow {\mathbb C}^*$, with $\mu (e_1)=1$, denote
$\mu (e_i)=\mu _i$ for all $i\in \{2, ..., n\}$.
For $c_1, ..., c_n\in {\mathbb C}$, denote by $circ (c_1, ..., c_n; \mu _2, ..., \mu _n)$
the $n\times n$ matrix with $c_1, c_2, ..., c_n$ in the first row,
$c_1$ on the main diagonal and entry $c_{j-i+1}\frac{\mu _i\mu _{j-i+1}}
{\mu _j}$ in any other position $(i, j)$ (we put
$\mu _1=1$). Certainly, $circ (c_1, ..., c_n)=
circ (c_1, ..., c_n; 1, ..., 1)$, so the matrices of the type
$circ (c_1, ..., c_n; \mu _2, ..., \mu _n)$
generalize circulant matrices.
They form an algebra, denoted by $C^n_{{\mathbb C}}(\mu )$, which is
isomorphic to the twisted group algebra with two-cocycle induced by $\mu $ (this
is actually how we arrived at these matrices). From general results we know that this
twisted group ring is isomorphic to the ordinary group ring, so $C^n_{{\mathbb C}}(\mu )$
is isomorphic to the algebra of $n\times n$ circulants. But these matrices generalize
not only circulants, but also skew circulant matrices: indeed,
a skew circulant matrix $scirc (c_1, ..., c_n)$ turns out to be
$circ (c_1, ..., c_n; \sigma, \sigma ^2, ..., \sigma ^{n-1})$, where
$\sigma =cos (\frac{\pi }{n})+isin (\frac{\pi }{n})$. Moreover,
for a matrix $circ (c_1, ..., c_n; \mu _2, ..., \mu _n)$ we are able to give an explicit formula for its
eigenvalues (and eigenvectors) that can be read directly from its entries;
this formula is a common generalization for the formulae (cf. \cite{davis}) giving the
eigenvalues for circulant and skew circulant matrices.
\section{Preliminaries}
\setcounter{equation}{0}
${\;\;\;\;}$We work over a base field ${\mathbb K}$ of characteristic zero (so all
matrices have entries in ${\mathbb K}$ and all algebras and Hopf
algebras are over ${\mathbb K}$). Sometimes we will need to choose
${\mathbb K}$ to be the field ${\mathbb C}$ of complex numbers, and we will specify
this every time we do it.

When we work with $n\times n$ circulant matrices or with elements in the
cyclic group of order $n$, all subscripts will be considered mod n. For a
matrix $A$ we denote by $A^T$ its transpose.

A circulant matrix of order n is an $n\times n$ matrix with the property
that each row is obtained from the previous row by rotating once to the right.
Obviously such a matrix is completely determined by its first row. If the first row is
$(c_1, c_2, ..., c_n)$ we will denote the associated circulant matrix by
$circ (c_1, c_2, ..., c_n)$. Note that the $(i, j)$ entry of this matrix is
$c_{j-i+1}$. For instance, for $n=3$, the matrix $circ(c_1, c_2, c_3)$ is

\[ \left( \begin{array}{ccc}
c_1 & c_2 & c_3\\
c_3 & c_1 & c_2\\
c_2 & c_3 & c_1
\end{array} \right)\]

If we denote by $C_{\mathbb K}^n$ the set of circulant $n\times n$ matrices, then
$C_{\mathbb K}^n$ is an $n$ dimensional subalgebra of the algebra
$M_n({\mathbb K})$ of $n\times n$ matrices over ${\mathbb K}$.

If ${\mathbb K}={\mathbb C}$, the field of complex numbers, and
$C=circ (c_1, c_2, ..., c_n)$ is a complex circulant matrix, then the eigenvalues and
eigenvectors of $C$ may be written down explicitely (cf. \cite{davis}): if we denote by
$p_C(X)=c_1+c_2X+...+c_nX^{n-1}$ and
$\omega =cos (\frac{2\pi }{n})+i sin (\frac{2\pi }{n})$, then the eigenvalues of $C$
are the scalars $\lambda _j=p_C(\omega ^{j-1})$, for $1\leq j\leq n$, the eigenvector
of $\lambda _j$ being the vector $x_j=(1, \omega ^{j-1}, \omega ^{2(j-1)}, ..., \omega ^{(n-1)(j-1)})^T$.

For more properties of circulant matrices we refer to \cite{davis}, while for
terminology, notation etc. concerning Hopf algebras we refer to
\cite{mont}.
\section{Forms associated to circulant matrices}
\setcounter{equation}{0}
${\;\;\;\;}$
We consider the group ${\mathbb Z}_n$ (cyclic group of order $n$) and we denote
its elements by $e_i=\widehat{i-1}$, for $i\in \{1, 2, ..., n\}$, so we have
$e_ie_j=e_{i+j-1}$ and $e_i^{-1}=e_{n-i+2}$, for all
$i, j\in \{1, 2, ..., n\}$ (subscripts mod n, according to our convention).

We consider the group algebra
${\mathbb K}{\mathbb Z}_n$ over the base field ${\mathbb K}$. The following result is well known:
\begin{theorem}\label{isom}
The group algebra ${\mathbb K}{\mathbb Z}_n$ and the algebra $C_{\mathbb K}^n$
of circulant $n\times n$ matrices over ${\mathbb K}$ are isomorphic, an explicit isomorphism being defined by
$g:{\mathbb K}{\mathbb Z}_n\rightarrow C_{\mathbb K}^n$, $g(c_1e_1+...+c_ne_n)=
circ (c_1, c_2, ..., c_n)$, for all $c_1, ..., c_n\in {\mathbb K}$.
\end{theorem}
\begin{corollary} \label{defforms}
Every element $x=c_1e_1+...+c_ne_n\in {\mathbb K}{\mathbb Z}_n$ is a
solution to a certain polynomial $X^n-q_1(x)X^{n-1}+...+(-1)^nq_n(x)$.
Here $q_1$ is a linear form called the trace of $x$ and $q_n$ is an $n$-form
called the norm of $x$. Moreover, $q_1(x)=tr[circ (c_1, ..., c_n)]=nc_1$
and $q_n(x)=det [circ (c_1, ..., c_n)]$.
\end{corollary}
\begin{proof}
The polynomial is exactly the characteristic polynomial of the circulant matrix
corresponding to $x$ via the isomorphism $g:{\mathbb K}{\mathbb Z}_n\simeq C_{\mathbb K}^n$.
\end{proof}
\begin{remark} \label{transf}
Via the isomorphism ${\mathbb K}{\mathbb Z}_n\simeq C_{\mathbb K}^n$, we can also
consider the forms $q_1, ..., q_n$ as being defined on $C_{\mathbb K}^n$.
\end{remark}
\begin{definition}
For $x\in {\mathbb K}{\mathbb Z}_n$, we define the {\em conjugate} of $x$ by the formula
\begin{eqnarray}
&&\overline{x}=(-1)^{n+1}x^{n-1}+(-1)^nq_1(x)x^{n-2}+...+q_{n-1}(x)\in
{\mathbb K}{\mathbb Z}_n. \label{conjugate}
\end{eqnarray}
\end{definition}

As a consequence of Theorem \ref{isom} and Corollary \ref{defforms}, we
immediately obtain:
\begin{corollary}
An element $x\in {\mathbb K}{\mathbb Z}_n$ is invertible if and only if
$q_n(x)\neq 0$, and in this case its inverse is defined by the formula
$x^{-1}=\frac{\overline{x}}{q_n(x)}$.
\end{corollary}
\begin{corollary}
An element $x=c_1e_1+...+c_ne_n\in {\mathbb C}{\mathbb Z}_n$ is
noninvertible if and only if there exists $y$ a root of unity of order $n$
such that $c_1+c_2y+...+c_ny^{n-1}=0$.
\end{corollary}
\begin{proof}
We know that $q_n(x)=det[circ (c_1, ..., c_n)]$, and since we are over ${\mathbb C}$ here
we know that $q_n(x)$ is the
product of the eigenvalues of $circ (c_1, ..., c_n)$, which are
given, as we mentioned before, by $\{p_C(\omega ^{j-1})\}$, for all $1\leq j\leq n$,
where $p_C(X)=c_1+c_2X+...+c_nX^{n-1}$ and
$\omega =cos (\frac{2\pi }{n})+i sin (\frac{2\pi }{n})$, so the result follows.
\end{proof}

We will present explicit formulae for the forms $q_i$ for $n=3$ and $n=4$.
The case $n=3$ was studied in detail in \cite{ak}, where the formulae for $q_1, q_2, q_3$
have been found. Namely, for $x=c_1e_1+c_2e_2+c_3e_3\in {\mathbb K}{\mathbb Z}_3$,
we have:
\begin{eqnarray*}
&&q_1(x)=3c_1, \;\;q_2(x)=3c_1^2-3c_2c_3, \;\;q_3(x)=c_1^3+c_2^3+
c_3^3-3c_1c_2c_3.
\end{eqnarray*}
Moreover, the following properties have been found in \cite{ak}:
\begin{proposition} \label{n3}
If $x, y\in {\mathbb K}{\mathbb Z}_3$, then $q_1(\overline{x})=q_2(x)$ and
$q_2(x+y)=q_2(x)+q_2(y)+q_1(x)q_1(y)-q_1(xy)$.
\end{proposition}

Similarly, by performing explicit computations, one can show that for
$x=c_1e_1+c_2e_2+c_3e_3+c_4e_4\in {\mathbb K}{\mathbb Z}_4$, we have:
\begin{eqnarray*}
&&q_1(x)=4c_1, \\
&&q_2(x)=6c_1^2-4c_2c_4-2c_3^2, \\
&&q_3(x)=4c_1^3-8c_1c_2c_4-4c_1c_3^2+4c_2^2c_3+4c_3c_4^2, \\
&&q_4(x)=c_1^4-c_2^4+c_3^4-c_4^4-2c_1^2c_3^2-4c_1^2c_2c_4+
4c_1c_2^2c_3+4c_1c_3c_4^2+2c_2^2c_4^2-4c_2c_3^2c_4,
\end{eqnarray*}
and also by explicit computations using these formulae one can prove:
\begin{proposition} \label{n4}
If $x, y\in {\mathbb K}{\mathbb Z}_4$, then $q_1(\overline{x})=q_3(x)$ and
$q_2(x+y)=q_2(x)+q_2(y)+q_1(x)q_1(y)-q_1(xy)$.
\end{proposition}

Actually, the above results admit generalizations
to arbitrary $n$, at least over ${\mathbb C}$:
\begin{proposition}
If $x\in {\mathbb C}{\mathbb Z}_n$, then $q_1(\overline{x})=q_{n-1}(x)$.
\end{proposition}
\begin{proof}
We denote by $s_0, s_1, s_2, ..., s_n$ the elementary symmetric polynomials in
$n$ variables $X_1, X_2, ..., X_n$, defined by $s_0=1$, $s_1=X_1+...+X_n$,
$s_2=\displaystyle\sum _{i<j}X_iX_j,..., s_n=X_1X_2...X_n$. Also, for any
$1\leq k\leq n$, we consider the polynomial $p_k$ defined by
$p_k=X_1^k+X_2^k+...+X_n^k$. These polynomials are related by
Newton's identities:
\begin{eqnarray*}
&&ks_k=\sum _{i=1}^k(-1)^{i-1}s_{k-i}p_i, \;\;\:\forall \;1\leq k\leq n.
\end{eqnarray*}
In particular, for $k=n-1$, we have
\begin{eqnarray}
&&(n-1)s_{n-1}=\sum _{i=1}^{n-1}(-1)^{i-1}s_{n-1-i}p_i. \label{Newton}
\end{eqnarray}
Let now $x=c_1e_1+...+c_ne_n\in {\mathbb C}{\mathbb Z}_n$ and consider the
circulant matrix $circ (c_1, ..., c_n)$, whose eigenvalues will be denoted
by $\lambda _1, ..., \lambda _n$. It is clear that, for any $1\leq i\leq n$, we have
$q_i(x)=s_i(\lambda _1, ..., \lambda _n)$. On the other hand, for any
$1\leq k\leq n$, the eigenvalues of the circulant matrix associated to $x^k$ are
$\lambda _1^k, ..., \lambda _n^k$ and so we have $q_1(x^k)=
\lambda _1^k+...+\lambda _n^k=p_k(\lambda _1, ..., \lambda _n)$.
Now we compute, directly from the formula (\ref{conjugate}):
\begin{eqnarray*}
q_1(\overline{x})&=&(-1)^{n+1}q_1(x^{n-1})+(-1)^nq_1(x)q_1(x^{n-2})
+...-q_{n-2}(x)q_1(x)+nq_{n-1}(x)\\
&=&(-1)^{n+1}p_{n-1}(\lambda _1, ..., \lambda _n)+
(-1)^ns_1(\lambda _1,..., \lambda _n)p_{n-2}(\lambda _1, ..., \lambda _n)+...+\\
&&
+s_{n-3}(\lambda _1, ..., \lambda _n)p_2(\lambda _1, ..., \lambda _n)
-s_{n-2}(\lambda _1, ..., \lambda _n)p_1(\lambda _1, ..., \lambda _n)
+ns_{n-1}(\lambda _1, ..., \lambda _n)\\
&\overset{(\ref{Newton})}{=}&s_{n-1}(\lambda _1, ..., \lambda _n)
=q_{n-1}(x),
\end{eqnarray*}
finishing the proof.
\end{proof}
\begin{proposition}
If $x, y\in {\mathbb C}{\mathbb Z}_n$, then
$q_2(x+y)=q_2(x)+q_2(y)+q_1(x)q_1(y)-q_1(xy)$.
\end{proposition}
\begin{proof}
We denote by $\lambda _1, ..., \lambda _n$ (respectively $\mu _1, ..., \mu_n$)
the eigenvalues of the circulant matrix corresponding to $x$ (respectively y). Then
the eigenvalues of the circulant matrix corresponding to $x+y$
(respectively $xy$) are $\lambda _1+\mu _1, ..., \lambda _n+\mu _n$
(respectively $\lambda _1\mu _1, ..., \lambda _n\mu _n$). So, we have:
\begin{eqnarray*}
q_2(x+y)&=&\sum _{i<j}(\lambda _i+\mu _i)(\lambda _j+\mu _j)\\
&=&\sum _{i<j}\lambda _i\lambda _j+\sum _{i<j}\mu _i\mu _j
+\sum _{i<j}\lambda _i\mu _j+\sum _{i<j}\mu _i\lambda _j\\
&=&q_2(x)+q_2(y)+\sum _{i<j}\lambda _i\mu _j+\sum _{i<j}\mu _i\lambda _j,
\end{eqnarray*}
\begin{eqnarray*}
q_1(x)q_1(y)-q_1(xy)&=&(\sum _{i=1}^n\lambda _i)(\sum _{j=1}^n\mu _j)
-\sum _{k=1}^n\lambda _k\mu _k\\
&=&\sum _{i<j}\lambda _i\mu _j+\sum _{i<j}\mu _i\lambda _j,
\end{eqnarray*}
hence indeed we have $q_2(x+y)=q_2(x)+q_2(y)+q_1(x)q_1(y)-q_1(xy)$.
\end{proof}
\section{Hopf-algebraic properties of circulant matrices}
\setcounter{equation}{0}
${\;\;\;\;}$
Let $H$ be a finite dimensional Hopf algebra. We denote as usual by
$\rightharpoonup $ the left regular action of $H$ on $H^*$ defined by
$(h\rightharpoonup \varphi )(h')=\varphi (h'h)$, for all $h, h'\in H$ and
$\varphi \in H^*$. It is well known that with this action $H^*$ becomes a
left $H$-module algebra, so we can consider the smash product
$H^*\#H$ (which is sometimes called in the literature the
{\em Heisenberg double} of $H^*$). It is also
well known (see \cite{mont}, p. 162)
that $H^*\#H$ is isomorphic as an algebra to the endomorphism algebra
$End(H^*)$, an explicit algebra isomorphism being defined by
\begin{eqnarray*}
&&\lambda :H^*\#H\simeq End(H^*), \;\;\;
\lambda (\varphi \# h)(\psi )=\varphi *
(h\rightharpoonup \psi ),
\end{eqnarray*}
for all $\varphi, \psi \in H^*$ and $h\in H$, where $*$ is the
convolution product
in $H^*$ defined by $(\varphi *\psi )(h)=\varphi (h_1)\psi (h_2)$,
where we used
the Sweedler-type notation $\Delta (h)=h_1\otimes h_2$ for the
comultiplication of $H$. In particular, the restrictions of $\lambda $ to
$H^*$ and $H$ define embeddings of $H^*$ and $H$ into
$End(H^*)$ (as algebras), defined respectively by
\begin{eqnarray*}
&&\lambda _{H^*}:H^*\rightarrow End(H^*), \;\;\;
\lambda _{H^*}(\varphi )(\psi )=
\varphi *\psi , \\
&&\lambda _H:H\rightarrow End(H^*), \;\;\;\lambda _H(h)(\psi )=
h\rightharpoonup \psi .
\end{eqnarray*}

If we identify $End(H^*)$ with $M_n({\mathbb K})$, where $n=dim (H)$, 
we obtain algebra embeddings of
$H$ and $H^*$ into $M_n({\mathbb K})$.
Our aim is to see how these embeddings look like if we take
$H={\mathbb K}{\mathbb Z}_n$ with its usual Hopf algebra structure.
\begin{proposition}
The
image of $({\mathbb K}{\mathbb Z}_n)^*$ in $M_n({\mathbb K})$ via the above
embedding is the algebra $D_{{\mathbb K}}^n$ of diagonal $n\times n$ matrices.
The
image of ${\mathbb K}{\mathbb Z}_n$ in $M_n({\mathbb K})$ via the
above embedding is the algebra $C_{{\mathbb K}}^n$ of circulant
matrices and the embedding $\lambda _{{\mathbb K}{\mathbb Z}_n}$
coincides with the algebra isomorphism  $g:{\mathbb K}{\mathbb Z}_n
\simeq C_{{\mathbb K}}^n$ defined in Theorem \ref{isom}.
\end{proposition}
\begin{proof}
We consider the basis $\{e_1, ..., e_n\}$ in ${\mathbb K}{\mathbb Z}_n$
as before and its
dual basis $\{p_1, ..., p_n\}$ in $({\mathbb K}{\mathbb Z}_n)^*$,
defined by $p_i(e_j)=\delta _{ij}$ for all $i, j\in \{1, ..., n\}$. Thus, the
algebra structure of $({\mathbb K}{\mathbb Z}_n)^*$ is defined on
this basis by
$p_ip_j=\delta _{ij}p_j$ and $p_1+...+p_n=1$. Moreover, we identify
$End(({\mathbb K}{\mathbb Z}_n)^*)\equiv M_n({\mathbb K})$ via
this basis, that is, if $f\in End(({\mathbb K}{\mathbb Z}_n)^*)$ and
$f(p_j)=\sum _{i=1}^na_{ij}p_i$, we identify $f$ with the matrix
$M_f=(a_{ij})_{1\leq i, j\leq n}$.

By using the formula for $\lambda _{H^*}$ given above, we can easily see that
$\lambda _{({\mathbb K}{\mathbb Z}_n)^*}(p_k)(p_j)=p_kp_j=\delta _{jk}p_k$,
for all
$1\leq j, k\leq n$, that is $\lambda _{({\mathbb K}{\mathbb Z}_n)^*}(p_k)$
coincides
via the identification $End(({\mathbb K}{\mathbb Z}_n)^*)\equiv
M_n({\mathbb K})$  with
the matrix having $1$ in the $(k, k)$ position and $0$ elsewhere.
Thus, for an arbitrary
element $x=d_1p_1+...+d_np_n\in ({\mathbb K}{\mathbb Z}_n)^*$, with
$d_1, ..., d_n\in {\mathbb K}$, we have
$M_{\lambda _{({\mathbb K}{\mathbb Z}_n)^*}(x)}=
diag (d_1, ..., d_n)$, the diagonal matrix with entries $d_1, ..., d_n$, q.e.d.

On the other hand, by using the formula for $\lambda _H$ given above,
we can see that
$\lambda _{{\mathbb K}{\mathbb Z}_n}(e_k)(p_j)=e_k\rightharpoonup p_j$,
for all $1\leq j, k\leq n$, and since we have $(e_k\rightharpoonup p_j)(e_i)=
p_j(e_ie_k)=p_j(e_{k+i-1})=\delta _{j, k+i-1}$, for all
$1\leq i, j, k\leq n$, we obtain $e_k\rightharpoonup p_j=p_{j-k+1 (mod \;n)}$,
that is $\lambda _{{\mathbb K}{\mathbb Z}_n}(e_k)(p_j)=p_{j-k+1 (mod \;n)}$,
for all $1\leq j, k\leq n$, and this means exactly that
$M_{\lambda _{{\mathbb K}{\mathbb Z}_n}(e_k)}=circ (0, 0, ..., 1, ..., 0)$,
where $1$ is in the $k^{th}$ position. Thus, for an arbitrary element
$y=c_1e_1+...+c_ne_n\in {{\mathbb K}{\mathbb Z}_n}$,
with $c_1, ..., c_n \in {\mathbb K}$, we have
$M_{\lambda _{{\mathbb K}{\mathbb Z}_n}(y)}=circ (c_1, ..., c_n)$, q.e.d.
\end{proof}

Let us recall some facts from \cite{cap}. If $X$ is an (associative unital)
algebra (with multiplication denoted by $x\otimes y\mapsto xy$ for all
$x, y\in X$)
and $A$, $B$ are subalgebras of $X$, we say that $X$ factorizes as
$X=AB$ if the map $A\ot B\rightarrow X$, $a\ot b\mapsto ab$, for all $a\in A$,
$b\in B$, is a linear isomorphism. This is equivalent to saying that there
exists
a so called {\em twisting map} $R:B\ot A\rightarrow A\ot B$ such that
$X$ is isomorphic as an algebra to the so called {\em twisted tensor product}
$A\ot _RB$.

It is well known that any smash product (such as $({\mathbb K}{\mathbb Z}_n)^*
\# {\mathbb K}{\mathbb Z}_n)$ is a particular case of a twisted
tensor product.
Since $({\mathbb K}{\mathbb Z}_n)^*\simeq D_{{\mathbb K}}^n$
and ${\mathbb K}{\mathbb Z}_n\simeq  C_{{\mathbb K}}^n$ as algebras,
we can conclude:
\begin{proposition}
The algebra of $n\times n$ matrices factorizes as $M_n({\mathbb K})=
 D_{{\mathbb K}}^n C_{{\mathbb K}}^n$.
\end{proposition}
\begin{remark} {\em
Assume that
${\mathbb K}={\mathbb C}$. In this case, it is known that the Hopf algebra
${\mathbb C}{\mathbb Z}_n$ is selfdual, that is
${\mathbb C}{\mathbb Z}_n$ is isomorphic as a Hopf algebra to
$({\mathbb C}{\mathbb Z}_n)^*$. If we consider
$\omega =cos (\frac{2\pi }{n})+i sin (\frac{2\pi }{n})$, an explicit
isomorphism $\phi :{\mathbb C}{\mathbb Z}_n\simeq
({\mathbb C}{\mathbb Z}_n)^*$ is defined by
$\phi (e_i)=\sum _{j=1}^n \omega ^{(i-1)(j-1)}p_j$, for all
$i\in \{1, ..., n\}$, where $\{e_1, ..., e_n\}$ is the standard basis in
${\mathbb C}{\mathbb Z}_n$ and $\{p_1, ..., p_n\}$ is its dual
basis in $({\mathbb C}{\mathbb Z}_n)^*$. We consider the algebra
isomorphisms ${\mathbb C}{\mathbb Z}_n\simeq C_{\mathbb{C}}^n$ and
$({\mathbb C}{\mathbb Z}_n)^*\simeq D_{\mathbb{C}}^n$ and thus we obtain an algebra isomorphism
$\psi :C_{\mathbb{C}}^n\simeq D_{\mathbb{C}}^n$. If
$C=circ (c_1, ..., c_n)$ is a circulant matrix, an easy computation
shows that $\psi (C)=diag (\lambda _1, ..., \lambda _n)$,
where $\{\lambda _1, ..., \lambda _n\}$ are the eigenvalues of $C$
defined by $\lambda _j=c_1+c_2\omega ^{j-1}+c_3\omega ^{2(j-1)}+...+
c_n\omega ^{(n-1)(j-1)}$, for all $j\in \{1, ..., n\}$. In particular,
this shows immediately the known fact that if $X, Y$ are circulant matrices with
eigenvalues $\lambda _1, ..., \lambda _n$ and respectively
$\beta _1, ..., \beta _n$, then the circulant matrix $X+Y$ has
eigenvalues $\lambda _1+\beta _1, ..., \lambda _n+\beta _n$ and the
circulant matrix $XY$ has eigenvalues
$\lambda _1\beta _1, ..., \lambda _n\beta _n$.}
\end{remark}

We denote by $P_n$ the {\em fundamental} circulant $n\times n$ matrix,
defined by $P_n=circ(0, 1, 0, ..., 0)$. It has the property that
$P_n^n=I_n$ and, if $C=circ(c_1, ..., c_n)$, then $C=c_1I_n+c_2P_n+
c_3P_n^2+...+c_nP_n^{n-1}$.
If we denote as before by $\{e_1, ..., e_n\}$ the standard basis of
${\mathbb K}{\mathbb Z}_n$, we obviously have $e_2^2=e_3$,
$e_2^3=e_4$, ..., $e_2^{n-1}=e_n$. Thus, since $g$ is an algebra map and
$g(e_2)=P_n$, we have, for all $2\leq i\leq n$,
$g(e_i)=g(e_2^{i-1})=P_n^{i-1}$. Recall also that, for any $1\leq i\leq n$,
the inverse of $e_i$ in ${\mathbb K}{\mathbb Z}_n$ is
$e_{n-i+2}$.

We look again at the algebra isomorphism $g:{\mathbb K}{\mathbb Z}_n
\simeq C_{{\mathbb K}}^n$. Since ${\mathbb K}{\mathbb Z}_n$ is a
Hopf algebra, we can transfer its structure to $C_{{\mathbb K}}^n$ via $g$,
and thus $C_{{\mathbb K}}^n$ becomes a Hopf algebra. We will write down
its counit, comultiplication and antipode.

Let $C=circ (c_1, ..., c_n)$ be a circulant matrix. It is easy to see that
the transferred counit is defined by
$\varepsilon :C_{{\mathbb K}}^n\rightarrow {\mathbb K}$,
$\varepsilon (C)=c_1+c_2+...+c_n$.

We compute now the comultiplication $\Delta :C_{{\mathbb K}}^n\rightarrow
C_{{\mathbb K}}^n\ot C_{{\mathbb K}}^n$:
\begin{eqnarray*}
\Delta (C)&=&(g\otimes g)\circ \Delta _{{\mathbb K}{\mathbb Z}_n}\circ
g^{-1}(C)\\
&=&(g\otimes g)\circ \Delta _{{\mathbb K}{\mathbb Z}_n}(c_1e_1+...
+c_ne_n)\\
&=&(g\otimes g)(c_1e_1\otimes e_1+...+c_ne_n\otimes e_n)\\
&=&c_1g(e_1)\otimes g(e_1)+...+c_ng(e_n)\otimes g(e_n)\\
&=&c_1I_n\otimes I_n+c_2P_n\otimes P_n+c_3P_n^2\otimes P_n^2+...+
c_nP_n^{n-1}\otimes P_n^{n-1}.
\end{eqnarray*}
Note that the counit property $(\varepsilon \otimes id)\circ \Delta =id$
applied to $C$ becomes, by using the above formulae and the fact that
$\varepsilon (P_n)=1$,
\begin{eqnarray*}
C&=&(\varepsilon \otimes id)\circ \Delta (C)\\
&=&(\varepsilon \otimes id)(c_1I_n\otimes I_n+c_2P_n\otimes P_n+
c_3P_n^2\otimes P_n^2+...+
c_nP_n^{n-1}\otimes P_n^{n-1})\\
&=&c_1I_n+c_2P_n+c_3P_n^2+...+c_nP_n^{n-1},
\end{eqnarray*}
that is, the counit property is equivalent to the basic property
of the fundamental circulant matrix.
We compute now the formula for the antipode $S:C_{{\mathbb K}}^n
\rightarrow C_{{\mathbb K}}^n$:
\begin{eqnarray*}
S(C)&=&g\circ S_{{\mathbb K}{\mathbb Z}_n}\circ g^{-1}(C)\\
&=&g\circ S_{{\mathbb K}{\mathbb Z}_n}(c_1e_1+...+c_ne_n)\\
&=&g(c_1e_1+c_2e_2^{-1}+...+c_ne_n^{-1})\\
&=&g(c_1e_1+c_2e_n+c_3e_{n-1}+...+c_ne_2)\\
&=&circ (c_1, c_n, c_{n-1}, ..., c_3, c_2)\\
&=&circ(c_1, c_2, ..., c_n)^T.
\end{eqnarray*}
That is, $S(C)$ is just the transpose of $C$.

If we denote by $Q_n=circ (0, 0, ...,0, 1)=P_n^T$ the transpose of
$P_n$, then the antipode property $S(C_{(1)})C_{(2)}=\varepsilon (C)I_n$,
with notation $\Delta (C)=C_{(1)}\otimes C_{(2)}$, is equivalent to the
following relation involving the matrices $P_n$ and $Q_n$:
\begin{eqnarray*}
&&c_1I_n+c_2Q_nP_n+c_3Q_n^2P_n^2+...+c_nQ_n^{n-1}P_n^{n-1}=
(c_1+...+c_n)I_n,
\end{eqnarray*}
which is obviously true because we actually have
$Q_nP_n=I_n$, that is $Q_n$ is the inverse of $P_n$.
\begin{remark}
It is well known that the element $x=\frac{1}{n}(e_1+...+e_n)$ is a so called
integral in ${\mathbb K}{\mathbb Z}_n$,
that is it satisfies the condition
$hx=\varepsilon (h)x$ for all $h\in {\mathbb K}{\mathbb Z}_n$. If
we write $h=c_1e_1+...+c_ne_n$, with $c_1, ..., c_n\in {\mathbb K}$,
then the equality $hx=\varepsilon (h)x$ may be transferred in
$C_{{\mathbb K}}^n$ via the isomorphism $g$, and we obtain
$g(h)g(x)=\varepsilon (h)g(x)$, that is
\begin{eqnarray*}
&&circ(c_1, ..., c_n)circ(1, 1, ..., 1)=(c_1+...+c_n)circ(1, 1, ..., 1),
\end{eqnarray*}
and this is equivalent to the fact that $c_1+...+c_n$ is an eigenvalue
for $circ(c_1, ..., c_n)$ with eigenvector $(1, 1, ..., 1)^T$.
\end{remark}
\section{Brandt algebras and lattices in circulat matrices}
\setcounter{equation}{0}
${\;\;\;\;}$ In this section we assume that the base field is ${\mathbb C}$.
\begin{theorem} \label{blockcirc}
Let $C=circ (c_1, ..., c_n)$ be a circulant matrix. We denote by
$\Delta $ the comultiplication of the Hopf algebra $C_{{\mathbb C}}^n$ and by
$P_n$ the fundamental circulant matrix $circ (0, 1, 0, ..., 0)$.
Then $\Delta (C)$, regarded as an $n^2\times n^2$ matrix, is the block
circulant with circulant blocks matrix $circ (c_1I_n, c_2P_n, ..., c_nP_n^{n-1})$.
Moreover, the eigenvalues of $\Delta (C)$ are the eigenvalues of $C$, each one
with multiplicity $n$.
\end{theorem}
\begin{proof}
The first statement follows easily from the formula
$\Delta (C)=c_1I_n\otimes I_n+c_2P_n\otimes P_n+c_3P_n^2\otimes P_n^2+...+
c_nP_n^{n-1}\otimes P_n^{n-1}$. To prove the second statement,
one verifies first that the block diagonal matrix that
diagonalizes $\Delta (C)$ is
\begin{eqnarray*}
&&diag (\Lambda _1+\Lambda _2+...+\Lambda _n,
\Lambda _1+\omega \Lambda _2+...+\omega ^{n-1}\Lambda _n, ...,
\Lambda _1+\omega ^{n-1}\Lambda _2+...+\omega ^{(n-1)^2}\Lambda _n),
\end{eqnarray*}
where $\omega =cos (\frac{2\pi }{n})+isin (\frac{2\pi }{n})$, and
\begin{eqnarray*}
&&\Lambda _1=c_1I_n, \;\;\;\Lambda _2=c_2diag (1, \omega, ..., \omega ^{n-1}),
..., \Lambda _n=c_ndiag (1, \omega ^{n-1}, ..., \omega ^{(n-1)^2}),
\end{eqnarray*}
and the eigenvalues of $\Delta (C)$ are still $p_C(1), p_C(\omega ), ...,
p_C(\omega ^{n-1})$, each one with multiplicity $n$.
\end{proof}

Recalling that the algebra of complex circulant matrices is isomorphic to ${\mathbb C}{\mathbb Z}_n$,
we can introduce the following analogue of the concept in \cite{ak}:
\begin{definition}
A set $B$ of complex circulant matrices is called an {\em integral (rational) Brandt algebra}
if
$q_i(a), q_i(b), q_i(a+b), q_i(ab)\in {\mathbb Z} ({\mathbb Q})$, for all $a, b\in B$ and $1\leq i\leq n$,
where $q_i$ are the forms defined in Corollary \ref{defforms}, transferred to $C^n_{\mathbb  C}$ as
in Remark \ref{transf}.
\end{definition}
\begin{proposition}
The set of complex circulant matrices that have integral (rational) eigenvalues
is an integral (rational) Brandt algebra.
\end{proposition}
\begin{proof}
Follows immediately from the fact that the eigenvalues of a sum (product) of
circulant matrices are the sums (products) of the eigenvalues of the given matrices and
by using the fact that for a circulant matrix $C$ all $q_i(C)$ are symmetric polynomials
in the eigenvalues of $C$.
\end{proof}
\begin{proposition}
The set of complex circulant matrices $C=circ (c_1, ..., c_n)$ that have integer (rational) eigenvalues
is given by
\begin{eqnarray*}
&&(c_1, c_2, ..., c_n)^T=\frac{1}{n}M(\lambda _1, ..., \lambda _n)^T,
\end{eqnarray*}
where $\lambda _1, ..., \lambda _n$ are integer (rational) numbers, and the $n\times n$ matrix $M$ has
$\overline{\omega ^{(i-1)(j-1)}}$ as the $(i, j)$ entry,
where $\omega =cos (\frac{2\pi }{n})+isin (\frac{2\pi }{n})$. Moreover, if the elements $\lambda _i$ satisfy
the extra condition $\lambda _{k+1}=\lambda _{n-k+1}$, for all $1\leq k\leq n-1$, then the matrix $C$
has real entries.
\end{proposition}
\begin{proof}
If we consider the matrix $A=(a_{ij})$ such that $a_{ij}=\omega ^{(i-1)(j-1)}$, it is easy to see that
$A(c_1, ..., c_n)^T=(p_C(1), p_C(\omega ), ..., p_C(\omega ^{n-1}))$, and the result follows because
the inverse of $A$ is $\frac{1}{n}M$. For the second statement, just note that in the matrix $M$, the
column $k+1$ is the conjugate of the column $n-k+1$.
\end{proof}

\begin{corollary}
The set of block circulants with circulant blocks $\Delta (B)$, where $B$ is the set of circulant
matrices with integral (rational) eigenvalues and $\Delta $ is the
comultiplication of the Hopf algebra $C_{{\mathbb C}}^n$,
is an integral (rational) Brandt algebra.
\end{corollary}
\begin{proof}
The result follows because, as we have seen in Theorem \ref{blockcirc}, the eigenvalues of
$\Delta (b)$, for any $b\in B$, coincide with the eigenvalues of $b$.
\end{proof}

Now we can study lattices in the algebra of complex circulant matrices. This study takes us to
some conditions about when a circulant matrix is obtained as integral linear combinations of some
elements of this algebra. Here we will focus to the subset of circulant matrices that have
integral entries.
\begin{theorem}
We consider the lattice ${\mathbb Z}v_1+{\mathbb Z}v_2+...+{\mathbb Z}v_n$ in the
algebra of complex circulant $n\times n$ matrices, generated by the linearly
independent vectors
\begin{eqnarray*}
&&v_1=c_{11}I_n+c_{12}P_n+...+c_{1n}P_n^{n-1}, \\
&&v_2=c_{21}I_n+c_{22}P_n+...+c_{2n}P_n^{n-1}, \\
&&..........................................................\\
&&v_n=c_{n1}I_n+c_{n2}P_n+...+c_{nn}P_n^{n-1}.
\end{eqnarray*}
We denote by $C$ the $n\times n$ matrix with entries $c_{ij}$. If all the entries
of the matrix $C^{-1}$ are integral, then any circulant matrix with integral entries
belongs to this lattice, that is, may be written as $a_1v_1+a_2v_2+...+a_nv_n$, with
$a_1, ..., a_n\in {\mathbb Z}$.
\end{theorem}
\begin{proof}
The matrix $C$ is invertible because we assumed that $v_1, ..., v_n$ are linearly
independent. Since we assumed that the entries of $C^{-1}$ are integral, it follows that each of
the matrices $I_n$, $P_n$, $P_n^2$, ..., $P_n^{n-1}$ may be written as an integral linear
combination of the elements $v_1, ..., v_n$, and now the result follows because any
circulant matrix $circ (c_1, ..., c_n)$ may be written as
$circ(c_1, ..., c_n)=c_1I_n+c_2P_n+
c_3P_n^2+...+c_nP_n^{n-1}$.
\end{proof}

Moreover, if we consider the lattice ${\mathbb Z}\Delta (v_1)+{\mathbb Z}\Delta (v_2)
+...+{\mathbb Z}\Delta (v_n)$ in the set of block circulants with circulant blocks, if
$C^{-1}$ has integral entries then every block circulant with circulant blocks
of the form $circ (a_1I_n, a_2P_n, ..., a_nP_n^{n-1})$, with $a_1, ..., a_n\in {\mathbb Z}$
belongs to this lattice, that is
may be written as an
integral linear combination of $\Delta (v_1), ..., \Delta (v_n)$.
\begin{example} {\em
Consider the lattice of $3\times 3$ circulant matrices
${\mathbb Z}v_1+{\mathbb Z}v_2+{\mathbb Z}v_3$, with
\begin{eqnarray*}
&&v_1=-P_3+P_3^{2}, \\
&&v_2=-\frac{1}{3}I_3+\frac{1}{3}P_3+\frac{1}{3}P_3^{2}, \\
&&v_3=\frac{1}{3}I_3+\frac{2}{3}P_3-\frac{1}{3}P_3^{2}.
\end{eqnarray*}
The matrix of coefficients
\[ \left( \begin{array}{ccc}
0 & -1 & 1\\
-1/3 & 1/3 & 1/3\\
1/3 & 2/3 & -1/3
\end{array} \right)\]
has an inverse with integral entries, so every circulant $3\times 3$ matrix
with integral entries may be
written as $av_1+bv_2+cv_3$, with
$a, b, c\in {\mathbb Z}$. Moreover, each matrix
$circ (a_1I_3, a_2P_3, a_3P_3^{2})$, with $a_1, a_2, a_3\in {\mathbb Z}$, may be
written as $m_1\Delta (v_1)+m_2\Delta (v_2)+m_3\Delta (v_3)$, with
$m_1, m_2, m_3\in {\mathbb Z}$.
}
\end{example}
\section{A class of generalized circulants}
\setcounter{equation}{0}
${\;\;\;\;}$There exist various generalizations of circulant matrices,
appeared around 1980, see for instance \cite{davis}, \cite{water} and references therein.
As proved by Waterhouse in \cite{water}, many of them are related to twisted group rings.
In this section we will introduce a certain class of matrices that generalize both
circulant and skew circulant matrices, also related to twisted group rings,
having the property that their eigenvalues can be read directly from the
entries of the matrices.

Consider again the cyclic group ${\mathbb Z}_n$ with elements denoted as before by
$e_1, ..., e_n$, and $F:{\mathbb Z}_n\times {\mathbb Z}_n\rightarrow
{\mathbb K}^*$ a two-cocycle (here ${\mathbb K}^*$ is the set of nonzero elements
in ${\mathbb K}$), that is $F$ satisfies
\begin{eqnarray*}
&&F(e_1, x)=F(x, e_1)=1, \;\;\; \forall \;x\in {\mathbb Z}_n, \\
&&F(x, y)F(xy, z)=F(y, z)F(x, yz), \;\;\; \forall \;x, y, z\in {\mathbb Z}_n.
\end{eqnarray*}
We can consider then the twisted group ring ${\mathbb K}_F{\mathbb Z}_n$,
which is an associative algebra (with unit $e_1$) obtained from the group ring
${\mathbb K}{\mathbb Z}_n$ by deforming its product using $F$, namely
$x\cdot _Fy=F(x, y)xy$, for all $x, y\in {\mathbb Z}_n$.

We have seen that ${\mathbb K}{\mathbb Z}_n$ may be embedded as an algebra in
$M_n({\mathbb K})$ (and the image of this embedding is the algebra of
circulant matrices $C^n_{{\mathbb K}}$). We can do something similar
for a twisted group ring ${\mathbb K}_F{\mathbb Z}_n$. Namely, define the map
$\lambda _F:{\mathbb K}_F{\mathbb Z}_n\rightarrow End (({\mathbb K}_F{\mathbb Z}_n)^*)$,
$\lambda _F(a)(\psi)=a\rightharpoonup \psi $, where
$(a\rightharpoonup \psi )(b)=\psi (ba)$, for all $a, b\in {\mathbb K}_F{\mathbb Z}_n$
and $\psi \in ({\mathbb K}_F{\mathbb Z}_n)^*$. If we denote by $\{p_j\}$ the basis of
$({\mathbb K}_F{\mathbb Z}_n)^*$ dual to the basis $\{e_j\}$ of
${\mathbb K}_F{\mathbb Z}_n$, one can check that we have
$\lambda _F(e_k)(p_j)=F(e_{j-k+1}, e_k)p_{j-k+1}$,
for all $j, k\in \{1,..., n\}$. If we identify as usual
$End (({\mathbb K}_F{\mathbb Z}_n)^*)\cong M_n({\mathbb K})$, we
obtain an algebra embedding $\lambda _F:{\mathbb K}_F{\mathbb Z}_n
\hookrightarrow M_n({\mathbb K})$, defined as follows: if
$x=c_1e_1+...+c_ne_n\in {\mathbb K}_F{\mathbb Z}_n$, then
$\lambda _F(x)$ is the $n\times n$ matrix whose $(i, j)$ entry
is $c_{j-i+1}F(e_i, e_{j-i+1})$. For instance, for $n=3$ and $n=4$ the
corresponding matrices are respectively given by

\[ \left( \begin{array}{ccc}
c_1 & c_2 & c_3\\
c_3F(e_2, e_3) & c_1 & c_2F(e_2, e_2)\\
c_2F(e_3, e_2) & c_3F(e_3, e_3) & c_1
\end{array} \right)\]

\[ \left( \begin{array}{cccc}
c_1 & c_2 & c_3 & c_4\\
c_4F(e_2, e_4) & c_1 & c_2F(e_2, e_2) & c_3F(e_2, e_3)\\
c_3F(e_3, e_3) & c_4F(e_3, e_4) & c_1 & c_2F(e_3, e_2)\\
c_2F(e_4, e_2) & c_3F(e_4, e_3) & c_4F(e_4, e_4) & c_1
\end{array} \right)\]

We denote by $C^n_{{\mathbb K}}(F)$ the image of the map
$\lambda _F$; so $C^n_{{\mathbb K}}(F)$ is an algebra,
isomorphic to ${\mathbb K}_F{\mathbb Z}_n$.

Assume now that the two-cocycle $F$ is trivial in the cohomology
group $H^2({\mathbb Z}_n, {\mathbb K}^*)$, that is there exists a
map $\mu : {\mathbb Z}_n\rightarrow {\mathbb K}^*$, with $\mu (e_1)=1$
(and with notation $\mu (e_i)=\mu _i$ for all $i\in \{2, ..., n\}$),
such that $F(e_i, e_j)=\mu (e_i)\mu (e_j) \mu (e_ie_j)^{-1}$,
for all $i, j\in \{1, ..., n\}$. We will denote the algebra
$C^n_{{\mathbb K}}(F)$ by $C^n_{{\mathbb K}}(\mu )$. Also,
for $x=c_1e_1+...+c_ne_n\in
{\mathbb K}_F{\mathbb Z}_n$, we will denote
\begin{eqnarray*}
&&\lambda _F(x):=\lambda _{\mu }(x)
:=circ (c_1, c_2, ..., c_n; \mu _2, ..., \mu _n),
\end{eqnarray*}
which is a matrix having $c_1, c_2, ..., c_n$ in the first row,
$c_1$ on the main diagonal and entry $c_{j-i+1}\frac{\mu _i\mu _{j-i+1}}
{\mu _j}$ in any other position $(i, j)$ (with the convention
$\mu _1=1$). Obviously, we have $circ (c_1, ..., c_n)=
circ (c_1, ..., c_n; 1, ..., 1)$.
For instance, for $n=3$, the matrix
$circ (c_1, c_2, c_3; a, b)$ is
\[ \left( \begin{array}{ccc}
c_1 & c_2 & c_3\\
c_3ab & c_1 & c_2a^2b^{-1}\\
c_2ab & c_3b^2a^{-1} & c_1
\end{array} \right)\]

We aim to find the eigenvalues and eigenvectors for a
matrix $circ (c_1, c_2, ..., c_n; \mu _2, ..., \mu _n)$.
For this, we will rely on the known fact that for the
two-cocycle $F$ given by $F(e_i, e_j)=\mu (e_i)\mu (e_j) \mu (e_ie_j)^{-1}$
the twisted group ring ${\mathbb K}_F{\mathbb Z}_n$ is
isomorphic to the group ring ${\mathbb K}{\mathbb Z}_n$,
an isomorphism being defined by $\varphi :{\mathbb K}_F{\mathbb Z}_n
\simeq {\mathbb K}{\mathbb Z}_n$, $\varphi (e_i)=\mu (e_i)e_i$ for all
$1\leq i\leq n$.  Thus, we have also an algebra isomorphism
$\Psi :C^n_{{\mathbb K}}(\mu )\simeq C^n_{{\mathbb K}}$ defined by
$\Psi (circ (c_1, ..., c_n;\mu _2, ..., \mu _n))=
circ (c_1, c_2\mu _2, c_3\mu _3, ..., c_n\mu _n)$ with inverse
$\Psi ^{-1}(circ (c_1, ..., c_n))=circ (c_1, \frac{c_2}{\mu _2}, ...,
\frac{c_n}{\mu _n}; \mu _2, ..., \mu _n)$.
\begin{proposition}
Assume that $\lambda $ is an eigenvalue for the matrix
$circ (c_1, c_2\mu _2, ..., c_n\mu _n)$,
with eigenvector $(x_1, x_2, ..., x_n)^T$. Then $\lambda $ is an eigenvalue for
$circ (c_1, c_2, ..., c_n; \mu _2,..., \mu _n)$ with eigenvector
$(x_1, x_2\mu _2, ..., x_n\mu _n)^T$.
\end{proposition}
\begin{proof}
In order to be able to use the algebra isomorphism $\Psi $, we need to
transform the equality
\begin{eqnarray}
&&circ (c_1, c_2\mu _2, ..., c_n\mu _n)(x_1, ..., x_n)^T=
\lambda (x_1, ..., x_n)^T  \label{cucu}
\end{eqnarray}
into a relation between circulant matrices. We consider the circulant
$circ (x_1, x_n, x_{n-1}, ..., x_3, x_2)$, and we remark
that the first column of
$circ (c_1, c_2\mu _2, ..., c_n\mu _n)circ (x_1, x_n, x_{n-1}, ..., x_3, x_2)$
is exactly $circ (c_1, c_2\mu _2, ..., c_n\mu _n)(x_1, ..., x_n)^T$, and by
(\ref{cucu}) this column is $(\lambda x_1, ..., \lambda x_n)^T$. On the
other hand, $circ (c_1, c_2\mu _2, ..., c_n\mu _n)circ (x_1, x_n, x_{n-1}, ..., x_3, x_2)$
is a circulant matrix (being the product of two circulants) and since we know
that its first column is $(\lambda x_1, ..., \lambda x_n)^T$ we find out that
$circ (c_1, c_2\mu _2, ..., c_n\mu _n)circ (x_1, x_n, x_{n-1}, ..., x_3, x_2)=
circ (\lambda x_1, \lambda x_n, \lambda x_{n-1}, ..., \lambda x_3, \lambda x_2)$.
To this equality we apply the algebra map $\Psi ^{-1}$; we obtain:
\begin{eqnarray*}
&&circ (c_1, c_2, ..., c_n; \mu _2, ..., \mu _n)
circ (x_1, \frac{x_n}{\mu _2}, \frac{x_{n-1}}{\mu _3}, ...,
\frac{x_3}{\mu _{n-1}}, \frac{x_2}{\mu _n}; \mu _2, ..., \mu _n)\\
&&\;\;\;\;\;\;\;\;\;=
circ (\lambda x_1, \frac{\lambda x_n}{\mu _2}, \frac{\lambda x_{n-1}}{\mu _3}, ...,
\frac{\lambda x_3}{\mu _{n-1}}, \frac{\lambda x_2}{\mu _n}; \mu _2, ..., \mu _n).
\end{eqnarray*}
The first column of the matrix $circ (x_1, \frac{x_n}{\mu _2}, \frac{x_{n-1}}{\mu _3}, ...,
\frac{x_3}{\mu _{n-1}}, \frac{x_2}{\mu _n}; \mu _2, ..., \mu _n)$ is
\begin{eqnarray*}
&&(x_1, \frac{x_2}{\mu _n}F(e_2, e_n), \frac{x_3}{\mu _{n-1}}F(e_3, e_{n-1}), ...,
\frac{x_{n-1}}{\mu _3}F(e_{n-1}, e_3), \frac{x_n}{\mu _2}F(e_n, e_2))^T,
\end{eqnarray*}
that is
\begin{eqnarray*}
&&(x_1, \frac{x_2}{\mu _n}\mu _2\mu _n, \frac{x_3}{\mu _{n-1}}\mu _3\mu _{n-1}, ...,
\frac{x_{n-1}}{\mu _3}\mu _{n-1}\mu _3, \frac{x_n}{\mu _2}\mu _n\mu _2)^T,
\end{eqnarray*}
which is $(x_1, x_2\mu _2, ..., x_{n-1}\mu _{n-1}, x_n\mu _n)^T$.
Similarly, we can see that the first column of the matrix
$circ (\lambda x_1, \frac{\lambda x_n}{\mu _2}, \frac{\lambda x_{n-1}}{\mu _3}, ...,
\frac{\lambda x_3}{\mu _{n-1}}, \frac{\lambda x_2}{\mu _n}; \mu _2, ..., \mu _n)$ is
$(\lambda x_1, \lambda x_2\mu _2, ..., \lambda x_n\mu _n)^T$.
So, the above matrix equality implies
$circ (c_1, ..., c_n; \mu _2, ..., \mu _n)(x_1, x_2\mu _2, ..., x_{n-1}\mu _{n-1}, x_n\mu _n)^T
=\lambda (x_1, x_2\mu _2, ..., x_n\mu _n)^T$,
i.e. $\lambda $ is an eigenvalue for $circ (c_1, c_2, ..., c_n; \mu _2,..., \mu _n)$
with eigenvector
$(x_1, x_2\mu _2, ..., x_n\mu _n)^T$.
\end{proof}

Assume now that ${\mathbb K}={\mathbb C}$.
Since we know (\cite{davis}) the eigenvalues and eigenvectors for any circulant
matrix, in particular for $circ (c_1, c_2\mu _2, ..., c_n\mu _n)$, we
obtain immediately from this Proposition the eigenvalues and eigenvectors for
$circ (c_1, c_2, ..., c_n; \mu _2,..., \mu _n)$:
\begin{proposition} \label{eigen}
For the complex matrix $C=circ (c_1, c_2, ..., c_n; \mu _2,..., \mu _n)$ define the
polynomial $p_C(X)=c_1+c_2\mu _2X+c_3\mu _3X^2+...+c_n\mu _nX^{n-1}$. For
$j\in \{1, 2,..., n\}$ define $\lambda _j=p_C(\omega ^{j-1})$, where
$\omega =cos (\frac{2\pi }{n})+i sin (\frac{2\pi }{n})$. Then $\lambda _1, ..., \lambda _n$ are the
eigenvalues of $C$, and the eigenvector of $\lambda _j$ is
$x_j=(1, \mu _2\omega ^{j-1}, \mu _3\omega ^{2(j-1)}, ..., \mu _n\omega ^{(n-1)(j-1)})^T$,
for all $1\leq j\leq n$.
\end{proposition}
\begin{remark}
Exactly as for ordinary circulant matrices, it follows that
every element $x=c_1e_1+...+c_ne_n\in {\mathbb K}_F{\mathbb Z}_n$ is a
solution to a certain polynomial $X^n-q_1(x)X^{n-1}+...+(-1)^nq_n(x)$, with
$q_1(x)=tr[circ (c_1, ..., c_n; \mu _2, ..., \mu _n)]=nc_1$
and $q_n(x)=det [circ (c_1, ..., c_n; \mu _2, ..., \mu _n)]$. Moreover, if
${\mathbb K}={\mathbb C}$, then for any $1\leq i\leq n$ we have that
$q_i(x)=s_i(\lambda _1,..., \lambda _n)$, where $s_i$ is the $i ^{th}$ elementary
symmetric polynomial and $\lambda _1, ..., \lambda _n$ are the eigenvalues of
$circ (c_1, ..., c_n; \mu _2, ..., \mu _n)$.
\end{remark}

We show now that, over the field ${\mathbb C}$, the matrices of the type
$circ (c_1, c_2, ..., c_n; \mu _2,..., \mu _n)$
generalize not only circulant matrices but also skew circulant matrices. Recall from
\cite{davis} that a skew circulant matrix is a circulant followed by a change in sign to all
the elements below the main diagonal. Such a matrix is denoted by $scirc (c_1, ..., c_n)$.
For example, $scirc (a, b, c)$ is the matrix
\[ \left( \begin{array}{ccc}
a & b & c\\
-c & a & b\\
-b & -c & a
\end{array} \right)\]

For a given $n$, we denote by $\sigma =cos (\frac{\pi }{n})+isin (\frac{\pi }{n})$ and
$\omega =\sigma ^2=cos (\frac{2\pi }{n})+isin (\frac{2\pi }{n})$. With this notation,
a straightforward computation (using the fact that $\sigma ^n=-1$) shows:
\begin{proposition}
$scirc (c_1, ..., c_n)=circ (c_1, ..., c_n; \sigma, \sigma ^2, ..., \sigma ^{n-1})$.
\end{proposition}

Consequently, the skew circulant matrices are a subalgebra of $M_n({\mathbb C})$ (this was noticed
also in \cite{davis}), which will be denoted by $sC^n_{{\mathbb C}}$. By what we have
done before it follows that we have an algebra isomorphism $\Psi :sC^n_{{\mathbb C}}
\simeq C^n_{{\mathbb C}}$, $\Psi (scirc (c_1, ..., c_n))=
circ (c_1, \sigma c_2, \sigma ^2c_3, ..., \sigma ^{n-1}c_n)$. Moreover,
the eigenvalues of a skew circulant matrix, computed in \cite{davis},
may be reobtained by applying Proposition \ref{eigen}: namely, the
eigenvalues of $scirc (c_1, ..., c_n)=circ (c_1, ..., c_n; \sigma, \sigma ^2, ..., \sigma ^{n-1})$
are given by $\lambda _j=p_C(\omega ^{j-1})$, for $1\leq j\leq n$, where
$p_C(X)=c_1+c_2\sigma X+c_3\sigma ^2X^2+...+c_n\sigma ^{n-1}X^{n-1}$.

\end{document}